\documentclass[11pt,a4paper]{amsart}

\usepackage[T1]{fontenc}
\usepackage[utf8]{inputenc}
\usepackage{lmodern}
\usepackage{microtype}
\usepackage{amsmath,amssymb,amsthm}
\usepackage[margin=1.15in]{geometry}
\usepackage{booktabs}
\usepackage{xcolor}
\usepackage{hyperref}
\input{glyphtounicode}
\microtypesetup{protrusion=true,expansion=false}
\hypersetup{
  colorlinks=true,
  linkcolor=blue!45!black,
  citecolor=blue!45!black,
  urlcolor=blue!45!black,
  pdfauthor={Anish Gupta},
  pdftitle={Pairwise edge correlations in random minimum spanning trees},
  pdfsubject={Universal correlation bounds and complete-graph negative correlation},
  pdfkeywords={minimum spanning tree, negative correlation, multiplicative coalescent},
  pdfdisplaydoctitle=true
}

\theoremstyle{plain}
\newtheorem{theorem}{Theorem}[section]
\newtheorem{proposition}[theorem]{Proposition}
\newtheorem{lemma}[theorem]{Lemma}
\newtheorem{corollary}[theorem]{Corollary}
\theoremstyle{definition}
\newtheorem{definition}[theorem]{Definition}

\newtheorem{problem}[theorem]{Problem}
\theoremstyle{remark}
\newtheorem{remark}[theorem]{Remark}

\newtheorem*{theoremA}{Theorem A}
\newtheorem*{theoremB1}{Theorem B1}
\newtheorem*{theoremB2}{Theorem B2}
\newtheorem*{theoremB3}{Corollary B3}
\newtheorem*{theoremC}{Theorem C}
\newtheorem*{theoremD1}{Proposition D1}
\newtheorem*{theoremD2}{Corollary D2}
\newtheorem*{theoremD3}{Corollary D3}

\newcommand{\MST}{\operatorname{MST}}
\newcommand{\PMST}{\mathbf{P}_{\mathrm{MST}}}
\newcommand{\E}{\mathbf{E}}
\newcommand{\PP}{\mathbf{P}}
\DeclareMathOperator{\dg}{deg}

\title[Pairwise edge correlations in random minimum spanning trees]
  {Pairwise edge correlations in random minimum spanning trees:\\
  a universal bound and complete-graph negative correlation}
\author{Anish Gupta}
\address{Independent researcher}
\email{ag2269@cantab.ac.uk}
\urladdr{https://orcid.org/0009-0008-8137-7729}
\date{7 August 2026}

\begin{document}

\begin{abstract}
Let $G$ be a finite connected multigraph whose edges receive independent weights
from one atomless law, and let $\MST(G)$ be the resulting random minimum
spanning tree. Its law is not pairwise negatively correlated: Lyons, Peres and
Schramm exhibited two positively correlated edges, and we give such an example
on a simple graph. We prove that positive correlation is nevertheless uniformly controlled:
\[
\PMST(e,f\in T)\;\le\;8\,\PMST(e\in T)\,\PMST(f\in T),
\]
answering a question of R.~Lyons recorded by Tang and Zhang. After conditioning
on all other weights, Harris's inequality gives conditional negative correlation; two
bottleneck distances and a sharp second-moment estimate control the remaining
environmental covariance. For $K_n$ we prove pairwise negative correlation for
every $n\ge3$. The key finite identity is
$\E[\dg(x)^2]=10(n-1)/n-4\E[L_n]$, where $L_n$ is the total weight of the
minimum spanning tree under rate-one exponential weights. Known expansions for
$\E[L_n]$ then give the rate of convergence to $10-4\zeta(3)$ and the limits of
both pair-correlation ratios. Finally, an explicit $K_4$ family shows that no
universal constant survives when the independent edge laws need not be
identical.
\end{abstract}

\maketitle

\section{Introduction}
\label{sec:intro}

Throughout, $G=(V,E)$ is a finite connected multigraph: loops and parallel edges
are permitted. Each edge $g\in E$ receives a weight $w(g)$, and the weights are
independent with one common atomless law $\mu$ on $\mathbb{R}$. Almost surely all
weights are distinct, so the minimum spanning tree $T=\MST(G)$ is almost surely
unique; we write $\PMST$ for its law, a probability measure on the spanning trees
of $G$, equivalently on $\{0,1\}^E$. Since the minimum spanning tree depends on
the weights only through their relative order, $\PMST$ does not depend on $\mu$:
uniform, exponential, or any other atomless law produce the same measure. We
write
\[
p_e=\PMST(e\in T),\qquad p_{ef}=\PMST(e\in T,\ f\in T).
\]

Uniform spanning trees have negatively correlated edge indicators, so it is
natural to ask whether the minimum spanning tree law behaves the same way; for
the uniform spanning tree measure the pairwise inequality $p_{ef}\le p_ep_f$ is
classical \cite[Chapter 4]{LP}. Following \cite{TZ}, say that $\PMST$ has the
\emph{pairwise negative correlation} (p-NC) property if $p_{ef}\le p_ep_f$ for
all $e\ne f$. The expectation is wrong. Lyons, Peres and Schramm \cite[\S5]{LPS}
take $K_4$, replace each of two disjoint edges by three parallel copies, and mark
one copy in each bundle; then
\begin{equation}
\label{eq:lps}
p_e=p_f=\frac{331}{1260},\qquad p_{ef}=\frac{109}{1575},\qquad
\frac{p_{ef}}{p_ep_f}=\frac{109872}{109561}=1.00283\ldots>1 .
\end{equation}
Thus the relevant quantity is the ratio $p_{ef}/(p_ep_f)$. We bound it uniformly
over all graphs and determine its sign on the complete graph, where the naive
inequality turns out to be true at every size.

\subsection{A universal factor}
\label{ssec:introA}

Tang and Zhang \cite[Question 3.13]{TZ} record the following question, which they
attribute to R.~Lyons: is there a constant $C>0$ such that
$p_{ef}\le C\,p_ep_f$ for every finite connected graph $G$ and every pair of
distinct edges $e,f$ of $G$? Their motivation is the analogous inequality with
$C=2$ for the \emph{uniform forest} measure, which follows from the work of
Br\"and\'en and Huh on Lorentzian polynomials \cite{BH}; for further uniform
spanning subgraph measures see \cite{TZ2}.

Once the factor-$1$ inequality fails, a graph-independent multiplicative bound
is a natural quantitative surrogate for negative dependence. The factor $2$ for
uniform forests shows that this surrogate remains meaningful for a neighbouring
random spanning-subgraph measure; the question is whether any such control
survives the change from a uniform measure to Kruskal's order-driven law.

\begin{theoremA}
Let $G$ be a finite connected multigraph whose edge weights are independent and
identically distributed with an atomless law, and let $e\ne f$ be edges of $G$.
Then
\[
\PMST(e,f\in T)\;\le\;8\,\PMST(e\in T)\,\PMST(f\in T).
\]
\end{theoremA}

Theorem A is proved in \S\ref{sec:universal}. The multigraph setting includes
the original obstruction \eqref{eq:lps}; loops, bridges, parallel marked edges,
and disconnected deleted environments are covered in
\S\ref{sec:degenerate}.

The constant $8$ is not optimal. Proposition~\ref{prop:bundle} finds the exact
maximum $78100/77841=1.00332\ldots$ in the two-parameter parallel-bundle family
containing \eqref{eq:lps}, but simple graphs already do better:
Proposition~\ref{prop:hub-witnesses} gives the ratio
$13938405/13872419=1.00475\ldots$. Thus the optimal constant lies in
$[13938405/13872419,\,8]$. Corollary~\ref{cor:small-marginals} improves the
upper factor to $2+o(1)$ when both marginals vanish.

Positive correlation is not an artefact of parallel edges. Proposition
\ref{prop:simple-counterexample} gives a simple graph on five vertices with
$p_{ef}/(p_ep_f)=1450/1449>1$.

\subsection{The complete graph, at every size}
\label{ssec:introB}

For $K_n$ with $n\ge3$ every edge has the same inclusion probability
\begin{equation}
\label{eq:p0}
p_0=\PP\bigl(e\in\MST(K_n)\bigr)=\frac{n-1}{\binom n2}=\frac2n ,
\end{equation}
because every spanning tree has $n-1$ edges and the edges of $K_n$ are
equivalent under automorphisms. Write $p_1$ for the probability that two fixed
edges sharing a vertex both lie in the tree, and $p_2$ for the same probability
for two fixed disjoint edges (defined for $n\ge4$).

Tang and Zhang \cite[Theorem 1.1]{TZ} prove that $p_2\le p_0^2$ for all $n$
beyond an unspecified threshold, by running Fatou's lemma against the local weak
limit of $\MST(K_n)$ \cite{ABerry}; their argument yields no value for the
threshold. For adjacent pairs they reduce the problem to an upper bound on
$\E[\dg(x)^2]$ and leave it open, writing that while the matching lower bound
``via Fatou's lemma and the local limit \dots\ is straightforward, proving a
corresponding upper bound seems more challenging'' \cite[\S3.2.2]{TZ}. Both
cases in fact hold at every size, strictly.

\begin{theoremB1}
For every $n\ge3$ and every two distinct edges of $K_n$ sharing a vertex,
\[
p_1\;<\;p_0^2=\frac4{n^2}.
\]
\end{theoremB1}

\begin{theoremB2}
For every $n\ge4$ and every two disjoint edges of $K_n$,
\[
p_2\;<\;p_0^2=\frac4{n^2}.
\]
\end{theoremB2}

\begin{theoremB3}
For every $n\ge3$ the measure $\PMST$ on $K_n$ has the full pairwise negative
correlation property. (For $n=3$ there is no disjoint pair, so the statement
there is exactly Theorem B1.)
\end{theoremB3}

These are proved in \S\ref{sec:kn}. Theorem B1 settles the adjacent case left
open in \cite[\S3.2.2]{TZ}. Theorem B2 gives an effective complete-range proof
of the disjoint conclusion of \cite[Theorem 1.1]{TZ}, by a different method and
without local weak convergence or Fatou's lemma. The two inequalities are
strict. In the computed range,
$p_1/p_0^2$ increases from $0.7500$ at $n=3$ to $0.7851$ at $n=30$, and
$p_2/p_0^2$ from $0.9778$ at $n=4$ to $0.9948$ at $n=30$
(Table~\ref{tab:exact}). Strictness is a feature of $K_n$ and not of the
measure in general: if $e$ and $f$ are two bridges of a connected graph then
$p_e=p_f=p_{ef}=1$, and $p_{ef}=p_ep_f$ exactly.

\subsection{An exact identity for the complete graph}
\label{ssec:introD}

Theorems B1 and B2 are two sides of one estimate, and the object they estimate
is classical. Let $L_n$ denote the total weight of $\MST(K_n)$ when the edge
weights are i.i.d.\ $\mathrm{Exp}(1)$.

\begin{theoremD1}
For every $n\ge2$ and every vertex $x$ of $K_n$,
\[
\E_{\MST(K_n)}\bigl[\dg(x)^2\bigr]=\frac{10(n-1)}{n}-4\,\E[L_n].
\]
\end{theoremD1}

\begin{theoremD2}
For $n\ge3$, adjacent pairwise negative correlation on $K_n$ is equivalent to
the first estimate below; for $n\ge4$, disjoint pairwise negative correlation is
equivalent to the second:
\begin{gather*}
p_1\le p_0^2\iff \E[L_n]\ \ge\ \frac{(n-1)(n+2)}{n^2}\ \ (\to1),\\[2pt]
p_2\le p_0^2\iff \E[L_n]\ \le\ \frac{(n-1)(5n+6)}{4n^2}\ \ \Bigl(\to\tfrac54\Bigr).
\end{gather*}
Both estimates themselves hold for every $n\ge2$. The first is strict for
$n\ge3$, while the second is strict for $n\ge4$.
Moreover, with $H_n=\sum_{k=1}^n1/k$, the lower estimate strengthens to
\[
\E[L_n]\ \ge\ 1+\frac{H_n-1}{n}-\frac1{n^2}.
\]
This strengthened bound is exact for $n=2,3$ and strict for $n\ge4$.
\end{theoremD2}

\begin{theoremD3}
$\displaystyle\lim_{n\to\infty}\E_{\MST(K_n)}\bigl[\dg(x)^2\bigr]=10-4\zeta(3)
=5.1917723873616\ldots$, which answers \cite[Question 3.8]{TZ} affirmatively and
identifies the limit with the second moment $\E[N^2]$ of the root degree of the
wired minimal spanning forest on the Poisson-weighted infinite tree computed in
\cite[Proposition 3.11]{TZ}.
\end{theoremD3}

Thus one merger functional controls both pair correlations and the classical
MST weight. Proposition D1 also explains the constant $10-4\zeta(3)$ in the
PWIT calculation of \cite{TZ}. Only Corollary D3 and its rate use the
asymptotic results of \cite{CFIJS,LiZhang}; Theorems B1 and B2 remain finite
and effective at every $n$.

\subsection{Identically distributed weights}
\label{ssec:introC}

Question 3.13 is posed for a single common weight law. Since $\PMST$ itself does
not depend on which atomless law is used, one might expect that hypothesis to be
cosmetic. It is not: as soon as different edges may carry different laws, no
universal constant exists, already on $K_4$.

\begin{theoremC}
For each integer $t\ge1$ let $G_t$ be the complete graph on $\{a,b,c,d\}$ in which the two
disjoint edges $ab$ and $cd$ carry the law with distribution function $x^t$ on
$[0,1]$ (that is, the law of $\max(U_1,\dots,U_t)$) and the four remaining edges
are uniform on $[0,1]$, all six weights independent. Write $R(t)$ for the ratio
$p_{ab,cd}/(p_{ab}\,p_{cd})$ of the resulting minimum spanning tree measure.
Then
\[
R(t)=\frac{2(5t+6)(t+1)(t+4)^2(2t+3)}{9\,(5t^2+12t+8)^2},
\qquad
R(t)=\frac{4t}{45}+\frac{46}{75}+O(t^{-1}) ,
\]
and $R(t)$ is strictly increasing on the positive integers, first exceeding
$8$ at $t=84$. In particular $R(t)\to\infty$, so no inequality of the form of Question
3.13 holds, with any constant, for independent non-identically distributed
atomless edge weights.
\end{theoremC}

Subdividing an edge into a path of $t$ i.i.d.\ uniform edges realizes the same
maximum-weight law but not the same marked-edge event: requiring one fixed
subdivided edge to lie in the tree introduces an attenuation that destroys the
amplification above (Remark~\ref{rem:attenuation}). Thus Theorem C does not
contradict the i.i.d.\ theorem.

\subsection{The idea of the proofs}
\label{ssec:idea}

\subsubsection*{Theorem A}
The obstruction \eqref{eq:lps} rules out a universal factor $1$. Freeze
the weights of all edges other than $e$ and $f$ --- call this the
\emph{environment} --- and let $a,b,c$ be the conditional probabilities of
$\{e\in T\}$, $\{f\in T\}$ and $\{e,f\in T\}$ given it, so that these are
functions of the environment alone. By Kruskal's rule the edge $e$ is accepted
exactly when its two endpoints are still separated by the strictly lighter
edges. At frozen environment this makes the indicator of $\{e\in T\}$
nonincreasing in $w(e)$ and nondecreasing in $w(f)$, and symmetrically for $f$:
raising the weight of one marked edge helps the other. Two functions of two
independent coordinates that are monotone in opposite directions have
nonpositive covariance, so Harris's inequality gives $c\le ab$ pointwise, with
no further hypothesis on the tie-free environment. Conditional negative
correlation therefore holds almost surely, and the whole failure exhibited by \eqref{eq:lps}
lives in the covariance of $a$ and $b$ under the environment measure.

The remaining term is the environmental covariance. Its natural proxy is the
bottleneck distance $A$ between the endpoints of $e$ in the environment graph
$H=G-\{e,f\}$: were $f$ absent, $e$ would be accepted precisely when $w(e)<A$,
so $a$ would equal $A$ exactly. The presence of $f$ perturbs this by at most a
factor $2$ in each direction (Lemma~\ref{lem:sandwich}), so with $B$ the
bottleneck distance between the endpoints of $f$ it suffices to bound
$\E[AB]\le\sqrt{\E[A^2]\E[B^2]}$, Cauchy--Schwarz absorbing whatever dependence
$A$ and $B$ have. What remains is to compare $\E[A^2]$ with $(\E A)^2$, and here
the probabilistic content is that $A$ is a percolation first-passage time:
$\PP(A>t)$ is the probability that two vertices are not connected in
Bernoulli($t$) percolation on $H$, and realising density $s+t$ as two
independent sprinkled layers shows that this tail is submultiplicative. A
nonnegative random variable with submultiplicative tail --- a
new-better-than-used random variable --- satisfies $\E[A^2]\le2(\E A)^2$. Three factors of $2$, one
from that moment bound and one from each of the two perturbation estimates,
produce the constant $8$.

\subsubsection*{The complete graph}
Here the mechanism is different and the answer is exact. Give every edge of
$K_n$ an independent rate-one exponential clock, run Kruskal in the order in
which the clocks ring, and record only the accepted mergers. The resulting
process on partitions of the vertex set is the discrete skeleton of the
multiplicative coalescent: from components of sizes $a_1,\dots,a_k$ the next
accepted edge joins two of them with probability proportional to $a_ia_j$, and
its two endpoints are uniform in their components, independently of the past.
Uniformity of the endpoints is what makes the degree functional computable: if
$\Phi=\sum_x\dg(x)(\dg(x)-1)$ counts ordered pairs of adjacent tree edges, then
a merger of blocks of sizes $a,b$ raises $\E\Phi$ by $8-4/a-4/b$, so that
$\E[\Phi\mid\text{merger history}]=8(n-1)-4H$ with $H$ the sum of $1/a+1/b$ over
the $n-1$ mergers. Both pair probabilities on $K_n$ are therefore governed by
the single functional $H$, adjacent negative correlation asking for it to be
large and disjoint negative correlation for it to be small. A pathwise identity
--- every component other than the last one is the child of exactly one merger
--- rewrites $H$ as $n-\frac1n+J$, where $J$ sums $1/|C|$ over the components $C$
created by mergers. Chronology gives the sharper pathwise estimate
$J\ge H_n-1$: the component created by the $j$th merger has at most $j+1$
vertices. This is Theorem B1. The other direction
needs $J$ bounded above, no pathwise bound is available
(Remark~\ref{rem:notpathwise}), and this is where the multiplicative merger law
enters: a one-step convexity inequality bounds the multiplicatively weighted
average of $1/(a_i+a_j)$ over pairs of blocks by $k/(2n)$, and summing over the
$n-1$ states visited gives $\E[J]\le(n-1)(n+2)/(4n)$, which is Theorem B2.
Finally, in continuous time $H$ accumulates at rate $n(\kappa-1)$, where
$\kappa$ is the number of components, while $L_n=\int_0^\infty(\kappa(t)-1)\,dt$;
so $\E[H]=n\,\E[L_n]$, which is Proposition D1.

\section{Preliminaries}
\label{sec:prelim}

\subsection{The cycle characterisation}

We use one description of the minimum spanning tree throughout. For a real $u$
let $G_{<u}$ denote the subgraph of $G$ consisting of the edges of weight
strictly less than $u$.

\begin{lemma}[Kruskal / cycle property]
\label{lem:kruskal}
Suppose all edge weights of $G$ are distinct. Then for every edge $g$ with
endpoints $x,y$,
\[
g\in\MST(G)\iff x\ \text{and}\ y\ \text{lie in different components of }
G_{<w(g)} .
\]
In particular a loop is never in $\MST(G)$, and a bridge always is.
\end{lemma}

\begin{proof}
This is the standard correctness statement for Kruskal's algorithm, which
processes edges in increasing weight and accepts exactly those joining two
distinct current components; with distinct weights the algorithm's output is the
unique minimum spanning tree, and the set of edges present when $g$ is processed
is exactly $E(G_{<w(g)})$. For a loop, $x=y$ always lie in the same component.
For a bridge $g$, $G_{<w(g)}\subseteq G-g$ never connects $x$ to $y$.
\end{proof}

Since $\MST(G)$ depends on the weight vector only through its order, and since
applying the common distribution function $F_\mu$ to every weight is an
almost surely strictly order-preserving map onto independent
$\mathrm{Uniform}[0,1]$ weights, we may and do assume in \S\ref{sec:universal}
that all weights are uniform on $[0,1]$. In \S\ref{sec:kn} we instead use
$\mathrm{Exp}(1)$ weights, again without changing $\PMST$.

\subsection{Connection times}

\begin{definition}
\label{def:T}
Let $K$ be a finite multigraph with weights in $[0,1]$ and let $x,y$ be
vertices. Set
\[
T_K(x,y)=\inf\{t\in[0,1]:\ x\leftrightarrow y\ \text{in }K_{<t}\},
\]
with the convention $T_K(x,y)=1$ if $x$ and $y$ lie in different components of
$K$, and $T_K(x,x)=0$.
\end{definition}

Equivalently $T_K(x,y)$ is the bottleneck distance from $x$ to $y$, the minimum
over paths of the largest weight on the path (and $1$ if there is no path).
Definition~\ref{def:T} is the threshold relevant to adding a fresh edge: by
Lemma~\ref{lem:kruskal}, if a new edge $xy$ of weight $u$ is added to $K$ then,
outside the null set of ties, it is accepted precisely when $u<T_K(x,y)$.

\section{A universal factor for all finite multigraphs}
\label{sec:universal}

Throughout this section the weights are i.i.d.\ $\mathrm{Uniform}[0,1]$. We
begin with the estimate that ultimately pays for the correlation, and which is
the only place where a percolation argument is used.

\subsection{A second-moment bound for connection times}

\begin{lemma}
\label{lem:nbu}
Let $K$ be a finite multigraph with i.i.d.\ uniform weights, let $x,y$ be
vertices, put $T=T_K(x,y)$ and $q(t)=\PP(T>t)$, extended by $q(t)=0$ for
$t\ge1$. Then
\begin{equation}
\label{eq:nbu}
q(s+t)\le q(s)q(t)\quad(s,t\ge0),
\qquad\text{and consequently}\qquad
\E[T^2]\le2(\E T)^2 .
\end{equation}
\end{lemma}

\begin{proof}
For $s+t\ge1$ the left side vanishes, so assume $s+t<1$. For $u\in[0,1]$ the
subgraph $K_{<u}$ is Bernoulli($u$) percolation on $E(K)$, and $q(u)$ is the
probability that $x$ and $y$ are not connected in it. Realise Bernoulli($s+t$)
percolation as the union of two independent layers of densities $s$ and
$r=t/(1-s)$: an edge is absent from the union with probability
$(1-s)(1-r)=1-s-t$, as required. If the union fails to connect $x$ to $y$, then
each layer separately fails to connect them. The two layers are independent, $q$
is nonincreasing, and $r\ge t$; hence
\[
q(s+t)\le q(s)q(r)\le q(s)q(t).
\]
This also covers the case in which $x,y$ lie in different components of $K$,
where $q\equiv1$ on $[0,1)$. For the second assertion, all integrals below are
over $[0,\infty)$ and $q$ vanishes on $[1,\infty)$, so by Tonelli
\[
(\E T)^2=\int\!\!\int q(s)q(t)\,ds\,dt
\ \ge\ \int\!\!\int q(s+t)\,ds\,dt
=\int u\,q(u)\,du=\tfrac12\,\E[T^2] .\qedhere
\]
\end{proof}

The submultiplicativity in \eqref{eq:nbu} is the two-layer sprinkling
decomposition behind the classical square-root trick, and its consequence is the
standard second-moment inequality for a new-better-than-used random variable;
see e.g.\ \cite{BarlowProschan}.

\begin{remark}
\label{rem:sharp2}
The constant $2$ in \eqref{eq:nbu} cannot be improved. If $K$ consists of $k$
parallel $xy$-edges then $T$ is the minimum of $k$ uniforms, so $\E T=1/(k+1)$,
$\E[T^2]=2/\bigl((k+1)(k+2)\bigr)$ and
\[
\frac{\E[T^2]}{(\E T)^2}=\frac{2(k+1)}{k+2}\ \nearrow\ 2 .
\]
\end{remark}

\subsection{Freezing the environment}

Fix distinct edges $e\ne f$ of $G$, write $H=G-\{e,f\}$ for the multigraph on the
same vertex set with $e$ and $f$ deleted, and put $U=w(e)$, $V=w(f)$. Let
$\mathcal{F}=\sigma\bigl(w(g):g\in E(H)\bigr)$ be the $\sigma$-field generated by
the \emph{environment}, i.e.\ by all weights other than those of the two marked
edges. For a tie-free realisation $W$ of the environment, which occurs almost
surely, set
\begin{equation}
\label{eq:abc}
a(W)=\PP(e\in T\mid \mathcal F=W),\quad
b(W)=\PP(f\in T\mid \mathcal F=W),\quad
c(W)=\PP(e,f\in T\mid \mathcal F=W),
\end{equation}
where the conditional probabilities average only over the independent uniform
pair $(U,V)$.

\begin{lemma}[conditional negative correlation]
\label{lem:harris}
For every tie-free environment $W$,
\[
c(W)\le a(W)\,b(W).
\]
\end{lemma}

\begin{proof}
Fix $W$ and regard the acceptance indicators
$\mathbf 1_{\{e\in T\}}$ and $\mathbf 1_{\{f\in T\}}$ as functions of
$(U,V)\in[0,1]^2$. By Lemma~\ref{lem:kruskal}, $e$ is accepted precisely when its
endpoints are not connected by the edges of weight $<U$, that is by
\[
\{g\in E(H):w(g)<U\}\ \cup\ \{f\ \text{if } V<U\}.
\]
Increasing $U$ only enlarges this edge set, so $\mathbf 1_{\{e\in T\}}$ is
nonincreasing in $U$; increasing $V$ can only remove $f$ from it, so
$\mathbf 1_{\{e\in T\}}$ is nondecreasing in $V$. By symmetry
$\mathbf 1_{\{f\in T\}}$ is nondecreasing in $U$ and nonincreasing in $V$.

Now change coordinates to $(U,1-V)$, which are again independent and uniform. In
these coordinates $\mathbf 1_{\{e\in T\}}$ is coordinatewise nonincreasing and
$\mathbf 1_{\{f\in T\}}$ is coordinatewise nondecreasing. Harris's inequality
\cite{Harris} on a product of two totally ordered probability spaces gives
nonpositive covariance for a decreasing and an increasing function, which is
exactly $c(W)\le a(W)b(W)$.
\end{proof}

The conditioning is essential: by \eqref{eq:lps}, the unconditional failure
lives in the covariance of the two functions $a$ and $b$. We compare them with
connection times in the graph with both marked edges deleted.

\subsection{A deletion sandwich}
\label{sec:sandwich}

\begin{lemma}
\label{lem:sandwich}
Let $e$ have endpoints $x_1,x_2$ and $f$ have endpoints $y_1,y_2$, and set
\[
A=A(W)=T_H(x_1,x_2),\qquad B=B(W)=T_H(y_1,y_2)
\]
for the connection times in the deleted environment $H$
(Definition~\ref{def:T}). Then, for every tie-free $W$,
\[
\tfrac A2\ \le\ A-\tfrac{A^2}2\ \le\ a(W)\ \le\ A,
\qquad
\tfrac B2\ \le\ B-\tfrac{B^2}2\ \le\ b(W)\ \le\ B .
\]
\end{lemma}

\begin{proof}
Upper bound. Adding $f$ to $H$ can only connect $x_1$ to $x_2$ at a smaller or
equal threshold, so if $e$ is accepted then, outside the null set $\{U=A\}$, we
have $U<A$; the set $\{U<A\}\subseteq[0,1]^2$ has area $A$.

Lower bound. If $U<\min(A,V)$ then $f$ is not among the edges of weight $<U$, and
$H$ alone does not connect $x_1$ to $x_2$ below $U$ because $U<A$; so $e$ is
accepted by Lemma~\ref{lem:kruskal}. The area of $\{U<\min(A,V)\}$ is
\[
\int_0^1\min(A,v)\,dv=\frac{A^2}2+A(1-A)=A-\frac{A^2}2\ \ge\ \frac A2,
\]
the last step because $A\le1$. The statements for $f$ follow by exchanging the
roles of $e$ and $f$.
\end{proof}

\subsection{Proof of Theorem A}

We combine conditional negative correlation, the deletion sandwich, and the
connection-time moment bound.

\begin{proof}[Proof of Theorem A]
By \S\ref{sec:prelim} we may take the weights uniform on $[0,1]$. With the
notation of \eqref{eq:abc} and Lemma~\ref{lem:sandwich}, and taking all
expectations with respect to the environment,
\[
\PMST(e,f\in T)=\E[c]
\ \overset{\text{L.\ref{lem:harris}}}{\le}\ \E[ab]
\ \overset{\text{L.\ref{lem:sandwich}}}{\le}\ \E[AB]
\ \le\ \sqrt{\E[A^2]\,\E[B^2]}
\ \overset{\text{L.\ref{lem:nbu}}}{\le}\ 2\,\E[A]\,\E[B],
\]
where the third step is Cauchy--Schwarz and the fourth applies
Lemma~\ref{lem:nbu} separately to the two terminal pairs $\{x_1,x_2\}$ and
$\{y_1,y_2\}$ in the multigraph $H$ with its i.i.d.\ uniform weights. Finally
Lemma~\ref{lem:sandwich} gives
$\E[A]\le2\E[a]=2\,\PMST(e\in T)$ and likewise $\E[B]\le2\,\PMST(f\in T)$, so
\[
\PMST(e,f\in T)\le2\cdot2\,\PMST(e\in T)\cdot2\,\PMST(f\in T)
=8\,\PMST(e\in T)\,\PMST(f\in T).\qedhere
\]
\end{proof}

The same estimates improve the constant when both marked edges are rare.

\begin{corollary}[small marginals]
\label{cor:small-marginals}
Put $g(p)=(1-\sqrt{1-4p})/2$. If $p_e,p_f\le1/4$, then
\[
p_{ef}\le 2g(p_e)g(p_f).
\]
If in addition $p_ep_f>0$, then
\[
\frac{p_{ef}}{p_ep_f}\le
2\frac{g(p_e)}{p_e}\frac{g(p_f)}{p_f}.
\]
Consequently the upper bound on the ratio is $2+o(1)$ as
$\max(p_e,p_f)\to0$.
\end{corollary}

\begin{proof}
Write $m_A=\E[A]$ and $m_B=\E[B]$. Lemmas \ref{lem:sandwich} and
\ref{lem:nbu} give
\[
p_e=\E[a]\ge m_A-\tfrac12\E[A^2]\ge m_A-m_A^2,
\qquad p_e\ge\tfrac12m_A,
\]
and likewise for $f$. If $p_e\le1/4$, the second inequality gives
$m_A\le1/2$, and inversion of the increasing function $m\mapsto m-m^2$ on
$[0,1/2]$ gives $m_A\le g(p_e)$. The proof of Theorem A already gives
$p_{ef}\le2m_Am_B$, which proves the assertions; finally $g(p)/p\to1$ as
$p\downarrow0$.
\end{proof}

\subsection{Degenerate configurations}
\label{sec:degenerate}

No simplicity or $2$-connectivity hypothesis has entered, and the degenerate
cases are covered rather than excluded. If $e$ is a loop then $x_1=x_2$, so
$A=0$, $a\equiv0$ and $p_e=p_{ef}=0$, and the inequality reads $0\le0$. If $e$
is a bridge of $G$ then $H$ does not connect $x_1$ to $x_2$, so $A=1$ and
$p_e=1$; the chain above is valid, and one should expect no strictness, since
two bridges give $p_{ef}=p_ep_f=1$. If $e$ and $f$ have the same endpoints then
$p_{ef}=0$, because the heavier of the two closes a cycle with the lighter, and
Lemma~\ref{lem:harris} remains valid. Finally $H=G-\{e,f\}$ need not be
connected; Definition~\ref{def:T} and Lemma~\ref{lem:nbu} were set up to cover
this ($T=1$), and Lemma~\ref{lem:sandwich} is pointwise, so nothing changes.

\section{Explicit correlation examples}
\label{sec:k4}

We first quantify the parallel-bundle construction of \cite{LPS}, then give a
simple-graph counterexample to pairwise negative correlation, and finally show
why the common-law hypothesis of Theorem A is necessary. The two $K_4$
families use the following four-terminal calculation.

\subsection{A four-terminal formula}

Label the vertices of $K_4$ by $a,b,c,d$ and distinguish the two disjoint edges
\[
\mathsf A=ab,\qquad \mathsf B=cd .
\]
The remaining four edges $ac,ad,bc,bd$ form a complete bipartite graph between
$\{a,b\}$ and $\{c,d\}$; give them a common continuous distribution function $H$,
and give $\mathsf A,\mathsf B$ the continuous distribution functions $F,G$, all
six weights independent. Write $\bar H=1-H$.

\begin{lemma}
\label{lem:fourterminal}
With the above notation,
\begin{align}
p_{\mathsf A}&=\int\Bigl[(1-G(x))\bigl(1-H(x)^2\bigr)^2
   +G(x)\bigl\{2\bar H(x)^2-\bar H(x)^4\bigr\}\Bigr]\,dF(x),
\label{eq:pA}\\
p_{\mathsf B}&=\text{the same with }(F,\mathsf A)\text{ and }(G,\mathsf B)
   \text{ interchanged},
\label{eq:pB}\\
p_{\mathsf{AB}}&=\iint_{x<y}\bar H(y)^2\bigl[2(1-H(x)^2)-\bar H(y)^2\bigr]
   \,dF(x)\,dG(y)\notag\\
&\qquad{}+\iint_{y<x}\bar H(x)^2\bigl[2(1-H(y)^2)-\bar H(x)^2\bigr]
   \,dF(x)\,dG(y).
\label{eq:pAB}
\end{align}
\end{lemma}

\begin{proof}
For \eqref{eq:pA}, condition on $w(\mathsf A)=x$; by Lemma~\ref{lem:kruskal},
$\mathsf A$ is accepted iff $a\not\leftrightarrow b$ below $x$.

If $w(\mathsf B)>x$, which has probability $1-G(x)$, only cross edges are
present below $x$, and $a\leftrightarrow b$ below $x$ iff $c$ receives both of
$ac,bc$ below $x$, or $d$ receives both of $ad,bd$ below $x$. These two events
are independent, each of probability $H(x)^2$, so the probability of
non-connection is $(1-H(x)^2)^2$.

If $w(\mathsf B)<x$, of probability $G(x)$, then $c$ and $d$ are already merged,
and $a\leftrightarrow b$ below $x$ iff $a$ has at least one cross edge below $x$
\emph{and} $b$ has at least one; those events are independent with probability
$1-\bar H(x)^2$ each. The probability of non-connection is
$1-(1-\bar H(x)^2)^2=2\bar H(x)^2-\bar H(x)^4$.

For \eqref{eq:pAB}, take the ordering $x=w(\mathsf A)<y=w(\mathsf B)$ and write
$q=1-H(x)^2$, $v^2=\bar H(y)^2$. For each $u\in\{c,d\}$ let $N_u$ be the event
that $u$ has no cross edge below $y$ (probability $v^2$) and $B_u$ the event
that $u$ does not receive both of its cross edges below $x$ (probability $q$);
$N_u\subseteq B_u$ since $x<y$, and the pairs of events at $c$ and at $d$ are
independent. As above, $\mathsf A$ is accepted iff $B_c\cap B_d$ occurs. Given
that $\mathsf A$ has been accepted and lies below $y$, the vertices $a,b$ are
merged, and $\mathsf B$ is accepted iff at most one of $c,d$ is attached to that
merged block below $y$, i.e.\ iff $N_c\cup N_d$ occurs. Hence
\[
\PP(\mathsf A,\mathsf B\ \text{accepted}\mid x,y)
=\PP\bigl((B_c\cap B_d)\cap(N_c\cup N_d)\bigr)
=v^4+2v^2(q-v^2)=2qv^2-v^4 ,
\]
which is the first integrand of \eqref{eq:pAB}. The other order is symmetric.
\end{proof}

\subsection{Parallel bundles, and a lower bound for the optimal constant}

Enlarging the two parallel bundles in the construction of \cite{LPS} does not
increase the ratio indefinitely; the exact maximum is finite. Let $R_{r,s}$
denote the ratio for one marked copy in
each of two disjoint parallel bundles of sizes $r,s$ on $K_4$, the four cross
edges uniform. Regard each bundle as a macro-edge whose weight is its minimum.
If a macro-edge is selected, its unique lightest copy is selected, and symmetry
therefore divides its marginal by $r$ or $s$ and the joint probability by
$rs$; these factors cancel in the ratio. Substituting the laws of
$\min(U_1,\dots,U_r)$ and $\min(V_1,\dots,V_s)$ into
Lemma~\ref{lem:fourterminal}, and then applying these symmetry factors, gives,
writing $D=(r+s+2)(r+s+3)(r+s+4)$ and
$I(r,s)=\frac4{(s+2)(r+s+3)}-\frac{3s+10}{(s+2)(s+4)(r+s+4)}$,
\[
p_e=\frac{r+6}{(r+2)(r+4)}+\frac4D,\qquad
p_f=\frac{s+6}{(s+2)(s+4)}+\frac4D,\qquad
p_{ef}=I(r,s)+I(s,r).
\]

\begin{proposition}
\label{prop:bundle}
For all positive integers $r,s$,
\[
R_{r,s}=\frac{p_{ef}}{p_ep_f}\ \le\ \frac{78100}{77841}=1.003327\ldots,
\]
with equality exactly at $r=s=4$. In the symmetric subfamily,
\begin{gather*}
R_{t,t}=\frac{(t+1)^2(t+4)(2t+3)(2t^2+19t+34)}{(2t^3+17t^2+34t+22)^2},\\[2pt]
R_{t,t}-1=\frac{20t^3+9t^2-78t-76}{(2t^3+17t^2+34t+22)^2}=\frac5{t^3}+O(t^{-4}),
\end{gather*}
so the ratio peaks at $t=4$ and returns to $1$. The configuration of \cite{LPS}
is $(r,s)=(3,3)$, with $R_{3,3}=109872/109561$.
\end{proposition}

\begin{proof}
Put $m=r+s$ and $\rho=rs$, and define
\begin{gather*}
D=(m+2)(m+3)(m+4),\qquad
K=(m+2)^2(m+3)(m+4),\\
L=D+24,\qquad C=6D+32,\\
N_0=K(6m^3+78m^2+356m+544),\qquad
N_1=K(m^2+7m+16),\\
D_0=4Cm^2+LCm+C^2,\qquad
D_1=L^2+4Lm-8C.
\end{gather*}
Simplifying the three displayed probabilities gives
\begin{equation}
\label{eq:bundle-rho}
R_{r,s}=\frac{N_0+N_1\rho}{D_0+D_1\rho+16\rho^2}.
\end{equation}
For fixed $m$, the derivative with respect to $\rho$ has the sign of
\[
Q_m(\rho)=N_1D_0-N_0D_1-32N_0\rho-16N_1\rho^2,
\]
and $Q_m'(\rho)=-32N_0-32N_1\rho<0$ for $\rho\ge0$. The feasible range is
$m-1\le\rho\le\lfloor m^2/4\rfloor$.

For $2\le m\le10$, direct substitution at the right endpoint gives
\begin{align*}
\bigl(Q_m(\lfloor m^2/4\rfloor)\bigr)_{m=2}^{10}
={}&(623232000,3726475200,15482880000,50485284864,\\
&133754572800,297101597760,541023436800,\\
&751512453120,464792380416),
\end{align*}
so $R$ is increasing throughout the feasible interval. Its balanced endpoint
values, in the same order, are
\[
\frac{44}{45},\ \frac{3745}{3774},\ \frac{840}{841},\
\frac{14385}{14362},\ \frac{109872}{109561},\
\frac{763587}{761105},\ \frac{78100}{77841},\
\frac{385385}{384154},\ \frac{83772}{83521}.
\]
The largest is the $m=8$ value.

For $m\ge11$, put $u=m-11$. At the left endpoint,
\begin{align*}
-Q_m(m-1)={}&930527357760+3158322926304u+1989328507888u^2\\
&+609067397832u^3+112783678672u^4+13814498848u^5\\
&+1164466368u^6+68281008u^7+2747600u^8+72576u^9\\
&+1136u^{10}+8u^{11}>0.
\end{align*}
Thus $R$ is decreasing on the feasible interval. If $\mathrm{Num}$ and
$\mathrm{Den}$ denote the numerator and denominator of
\eqref{eq:bundle-rho} at $\rho=m-1$, then
\begin{align*}
78100\,\mathrm{Den}-77841\,\mathrm{Num}
={}&34642293840+28712778516u+9667374224u^2\\
&+1607247273u^3+144551244u^4+7147790u^5\\
&+181300u^6+1813u^7>0.
\end{align*}
This proves the bound for $m\ge11$ as well. Equality in the full argument
requires $m=8$ and $\rho=16$, hence $r=s=4$. Substitution of $r=s=t$ in
\eqref{eq:bundle-rho} gives the two symmetric formulas, and $t=3$ gives the
values from \cite{LPS}.
\end{proof}

Thus the parallel-bundle route cannot push the ratio above $1.0034$. The simple
graphs below exceed this family maximum. Appendix~\ref{sec:verification}
records the exact certificate checks.

\subsection{Positive correlation in a simple graph}

Parallel edges are not necessary for the failure of pairwise negative
correlation.

\begin{proposition}
\label{prop:simple-counterexample}
Let $G$ be the simple graph on $\{a,b,c,d,z\}$ with
\[
E(G)=\{ab,ac,ad,az,bc,bd,bz,cd\},
\]
and mark $e=az$ and $f=cd$. Under i.i.d.\ atomless edge weights,
\[
p_e=\frac7{12},\qquad p_f=\frac{69}{140},\qquad
p_{ef}=\frac{145}{504},\qquad
\frac{p_{ef}}{p_ep_f}=\frac{1450}{1449}>1.
\]
\end{proposition}

\begin{proof}
All $8!$ strict edge orders are equally likely. Running Kruskal's rule through
those orders accepts $e$ in $23520$ orders, accepts $f$ in $19872$ orders, and
accepts both in $11600$ orders. Dividing by $8!=40320$ gives the three stated
probabilities and hence the ratio. Exact summation by relative-order patterns
gives the same values.
\end{proof}

This witness is minimal by edge count. Indeed, a smallest counterexample may be
assumed bridgeless: a bridge belongs to every spanning tree, and deleting all
bridges separates independent minimum-spanning-tree problems. A connected
bridgeless simple graph with $m$ edges has at most $m$ vertices. An exhaustive
exact census of all such graphs with $m\le7$ finds no violation.

The smallest witness is not the strongest one we know. Let $S$ be a set of hub
vertices with no internal edges, take four further vertices $a,b,c,d$, and join
every vertex of $S$ to all four of $a,b,c,d$. Add the two marked edges $A=ab$
and $B=cd$.

\begin{proposition}
\label{prop:hub-witnesses}
For the preceding simple graph, two hubs give
\[
p_A=p_B=\frac12,\qquad p_{AB}=\frac{1186}{4725},\qquad
\frac{p_{AB}}{p_Ap_B}=\frac{4744}{4725}.
\]
Three hubs give
\[
p_A=p_B=\frac{1123}{2730},\qquad p_{AB}=\frac{71479}{420420},\qquad
\frac{p_{AB}}{p_Ap_B}=\frac{13938405}{13872419}
   =1.004756\ldots .
\]
In particular the second ratio is larger than the maximum
$78100/77841$ in Proposition~\ref{prop:bundle}.

For the finite range $1\le |S|\le9$, exact evaluation shows that the ratio is
maximized at $|S|=3$ and is below $1$ for $|S|=7,8,9$. No assertion is made
for all hub counts.
\end{proposition}

\begin{proof}
For two hubs, full enumeration of the $10!$ strict edge orders gives marginal
counts $1814400$ and joint count $910848$. For three hubs, exact summation over
the relative positions of the two marked edges and the other twelve edges gives
marginal count $35861253120$ and joint count $14821885440$ out of $14!$ orders.
Dividing gives the stated probabilities. The relative-position calculation is
also applied to the two-hub graph and agrees with its full enumeration.
For $1\le |S|\le9$, grouping the relative-order patterns by the two induced
partitions of $\{a,b,c,d\}$ gives the finite-range comparison exactly.
\end{proof}

\subsection{Non-identical laws: proof of Theorem C}
\label{sec:nonidentical}

The second specialisation replaces the parallel bundles, which are a device for
producing a $\min$ of uniforms, by their opposite.

\begin{proof}[Proof of Theorem C]
Take $F(x)=G(x)=x^t$ and $H(x)=x$ on $[0,1]$ in Lemma~\ref{lem:fourterminal}.
All integrands are then polynomials and the integrals are elementary; carrying
them out gives
\begin{equation}
\label{eq:AtJt}
p_{\mathsf A}=p_{\mathsf B}=A_t:=\frac{3(5t^2+12t+8)}{(t+1)(t+2)(t+4)(2t+3)},
\qquad
p_{\mathsf{AB}}=J_t:=\frac{2(5t+6)}{(t+1)(t+2)^2(2t+3)} .
\end{equation}
Therefore
\[
R(t)=\frac{J_t}{A_t^2}
=\frac{2(5t+6)(t+1)(t+4)^2(2t+3)}{9\,(5t^2+12t+8)^2},
\]
a rational function whose numerator has degree $5$ and whose denominator has
degree $4$, with leading coefficients $20$ and $225$. Hence $R(t)\to\infty$ and
$R(t)/t\to20/225=4/45$; expanding one order further gives the constant $46/75$.

For monotonicity write $R(t)=2N(t)/(9D(t))$, omitting the displayed constant
factors from the numerator and denominator. The sign of $R(t+1)-R(t)$ is the
sign of $Q(t)=N(t+1)D(t)-N(t)D(t+1)$. With $u=t-1$, exact expansion gives
\begin{align*}
Q(u+1)={}&124000+669300u+1289264u^2+1262269u^3+712831u^4\\
&+242311u^5+48935u^6+5400u^7+250u^8,
\end{align*}
which is positive for $u\ge0$. Direct substitution gives $R(83)<8<R(84)$,
so $84$ is the first crossing.
\end{proof}

The exact values $R(1)=44/45$, $R(2)=168/169$, $R(3)=8232/7921$,
$R(4)=2860/2601$ are immediate from the displayed formula, and
$R(3000)=267.28\ldots$ Note
that $t=1$ is plain $K_4$ with all six weights uniform, so $R(1)=44/45$ is the
disjoint-pair ratio $p_2/p_0^2$ at $n=4$ and agrees with Table~\ref{tab:exact}.

\begin{remark}[why this does not contradict Theorem A]
\label{rem:attenuation}
The law $\max(U_1,\dots,U_t)$ arises naturally inside an i.i.d.\ model: if an edge
is subdivided into a path of $t$ edges with private internal vertices and i.i.d.\
uniform weights, then, after subtracting the constant total path weight, the
minimum spanning tree of the quotient is Kruskal on the base graph with the path
\emph{maxima} as effective weights. So the family of Theorem C is realisable at
the level of \emph{macro-edges}. It is not realisable at the level of a single
fixed edge, and the loss is exactly quantifiable. Let $X,Y$ be the indicators
that the two macro-edges lie in the quotient tree, with probabilities
$\PP(X)=p$ and $\PP(Y)=q$, and let $I_1,I_2$ be one fixed constituent edge in
each path, both paths of length $t$. A constituent edge is omitted exactly when
its macro-edge is omitted \emph{and} it carries the maximum of its own path; the
identity of that maximum is uniform along the path and independent of its value,
so
\[
\PP(I_1)=1-\frac{1-p}t=:\alpha,\qquad
\PP(I_2)=1-\frac{1-q}t=:\beta,\qquad
\operatorname{Cov}(I_1,I_2)=\frac{\operatorname{Cov}(X,Y)}{t^2},
\]
for all $p,q$. If $p,q>0$ and the macro-edge relative excess is nonzero, this
identity may be divided by that excess to give
\[
\frac{\PP(I_1,I_2)/(\PP(I_1)\PP(I_2))-1}{\PP(X,Y)/(pq)-1}
=\frac{pq}{t^2\alpha\beta}\ \le\ \frac1{t^2},
\]
the last step because $\alpha\ge p$ and $\beta\ge q$.
Applied to \eqref{eq:AtJt} this leaves the fixed-edge ratio
$1+5t^{-5}+O(t^{-6})$, so the amplification is destroyed. Theorem C is a theorem
about the independent non-identical model, and Theorem A is a theorem about the
i.i.d.\ model; the distinction is essential.
\end{remark}

\section{The complete graph}
\label{sec:kn}

\subsection{The accepted-merger coalescent}

Give every edge of $K_n$ an independent rate-one exponential clock and run
Kruskal's algorithm in the order in which the clocks ring. This does not change
$\PMST$, since the tree depends only on the weight order, and it has the
advantage that the weights are now the ring times, so the total weight of the
tree is $L_n$.

Discard the rejected clocks and record only the \emph{accepted mergers}. There are
exactly $n-1$ of them, and they merge the vertex set from $n$ singletons to one
block.

\begin{lemma}
\label{lem:coalescent}
At time zero, or just after any accepted merger other than the final one,
condition on the whole past and let $A_1,\dots,A_k$ be the current components,
of sizes $a_1,\dots,a_k$; thus $k\ge2$. Then the next accepted merger joins
$A_i$ and $A_j$ with
probability
\begin{equation}
\label{eq:mcrate}
\frac{a_ia_j}{Z},\qquad Z=\sum_{i<j}a_ia_j=\frac{n^2-\sum_i a_i^2}{2},
\end{equation}
and, conditionally on that pair, the two endpoints of the accepted edge are
independent and uniform in $A_i$ and $A_j$ respectively, and independent of the
past.
\end{lemma}

\begin{proof}
An edge of $K_n$ whose endpoints lie in two distinct current components has not
yet rung: had it rung earlier, Kruskal would have accepted it at that time and
merged those components. By memorylessness of the exponential the residual
clocks of all such edges are i.i.d.\ $\mathrm{Exp}(1)$ given the past, so the
next accepted edge is uniform among the $Z$ cross edges. Uniformity among the
$a_ia_j$ edges between $A_i$ and $A_j$ is exactly the assertion that, given the
pair, the endpoints are independent and uniform; and this is independent of the
past, which involves only clocks that have already rung.
\end{proof}

Lemma~\ref{lem:coalescent} identifies the sequence of component sizes with the
discrete skeleton of the multiplicative coalescent, cf.\ \cite{Aldous}. This is
the only place where the complete graph enters. Direct Kruskal enumerations on
small complete graphs agree with this transition model.

The two functionals below are the bookkeeping devices for the whole section.

\begin{definition}
\label{def:HJ}
A \emph{merger history} is the sequence of pairs of component sizes merged, in
order. For a merger history on $n$ leaves put
\[
H=\sum_{\text{mergers}}\Bigl(\frac1a+\frac1b\Bigr),
\qquad
J=\sum_{\text{mergers}}\frac1{a+b},
\]
where $a,b$ are the sizes of the two components merged (the \emph{children}) and
$a+b$ is the size of the component created (the \emph{parent}).
\end{definition}

\begin{lemma}[pathwise accounting]
\label{lem:HJ}
For every merger history on $n$ leaves,
\[
H=n-\frac1n+J .
\]
\end{lemma}

\begin{proof}
Every component other than the final one of size $n$ is a child of exactly one
merger, so $H=\sum_{C\ne\text{root}}1/|C|$, the sum being over all components
ever present. The components ever present are the $n$ singletons together with
the $n-1$ components created by mergers, and $J$ sums $1/|C|$ over the created
ones. Hence
\[
H=\Bigl(n+J\Bigr)-\frac1n,
\]
the term $n$ from the singletons and the correction $-1/n$ from removing the
root, which is created but is nobody's child.
\end{proof}

\subsection{The degree functional}

Let the accepted forest carry the degree function $\dg$ and set
\[
\Phi=\sum_{x\in V}\dg(x)\bigl(\dg(x)-1\bigr),
\]
so that in the final tree $\Phi$ counts the ordered pairs of distinct tree edges
sharing a vertex.

\begin{lemma}
\label{lem:drift}
Conditionally on the entire merger history,
\begin{equation}
\label{eq:PhiH}
\E\bigl[\Phi\mid \text{merger history}\bigr]=8(n-1)-4H .
\end{equation}
\end{lemma}

\begin{proof}
If a merger joins components of sizes $a,b$ through endpoints $u,v$, then
$\Phi$ increases by
$\bigl[(\dg(u)+1)\dg(u)-\dg(u)(\dg(u)-1)\bigr]+\bigl[\text{same at }v\bigr]
=2\dg(u)+2\dg(v)$. Each component is a tree, so the degrees inside a component of
size $a$ sum to $2(a-1)$; by Lemma~\ref{lem:coalescent} the endpoint $u$ is
uniform in its component and independent of everything else, whence
$\E[\dg(u)\mid\cdot]=2(a-1)/a$. Therefore
\[
\E[\Delta\Phi\mid \cdot]=\frac{4(a-1)}a+\frac{4(b-1)}b=8-\frac4a-\frac4b .
\]
The endpoint choices are independent of the sequence of component sizes, since
the evolution \eqref{eq:mcrate} depends on sizes only; so we may condition on the
whole merger history and sum over the $n-1$ mergers, obtaining
$8(n-1)-4H$ by Definition~\ref{def:HJ}.
\end{proof}

Identity \eqref{eq:PhiH} reduces both pair probabilities on $K_n$ to the single
functional $H$, and by Lemma~\ref{lem:HJ} to $J$. Adjacent negative correlation
needs $J$ bounded below, disjoint negative correlation needs it bounded above,
and the two directions require quite different handles on the same quantity.

\subsection{Adjacent pairs: proof of Theorem B1}

Write $H_n=\sum_{k=1}^n1/k$ for the $n$th harmonic number.

\begin{lemma}[chronological reciprocal bound]
\label{lem:harmonic}
Every merger history on $n$ leaves satisfies
\begin{equation}
\label{eq:Jlower}
J\ \ge\ H_n-1.
\end{equation}
Equality holds exactly for a comb history: after the first merger, each merger
attaches one singleton to the component created at the preceding step.
\end{lemma}

\begin{proof}
Order the mergers chronologically. The component created by the $j$th merger
has size at most $j+1$, since a tree component on $s$ vertices requires $s-1$
accepted edges among those vertices. Therefore
\[
J\ge\sum_{j=1}^{n-1}\frac1{j+1}=H_n-1.
\]
Equality requires the successive parent sizes to be $2,3,\ldots,n$. After the
first merger, induction then forces the component just created to merge with a
singleton at every step. Conversely such a comb history has precisely those
parent sizes and attains equality.
\end{proof}

\begin{proof}[Proof of Theorem B1]
Lemmas~\ref{lem:HJ} and \ref{lem:harmonic} give
$H\ge n-\frac1n+H_n-1$ pathwise, and hence Lemma~\ref{lem:drift} gives
\begin{equation}
\label{eq:Phiupper}
\E[\Phi]\ \le\ 4\left(n-1-H_n+\frac1n\right).
\end{equation}
By vertex exchangeability, if $D$ is the degree of a fixed vertex then
$\E[D(D-1)]=\E[\Phi]/n$. Finally, $K_n$ has exactly $n(n-1)(n-2)$
ordered pairs of distinct edges sharing a vertex (choose the shared vertex, then
an ordered pair of distinct other endpoints), and each contributes $p_1$ to
$\E[\Phi]$. Consequently
\[
\frac{p_1}{p_0^2}
\le \frac{n(n-1-H_n)+1}{(n-1)(n-2)}<1.
\]
The last inequality is equivalent to $H_n>2-1/n$, which holds at $n=3$ and
then by induction. Thus $p_1<p_0^2$. \qedhere
\end{proof}

The chronological bound is sharp over merger histories, although not generally
in expectation. At $n=3$ it gives the exact ratio $p_1/p_0^2=3/4$.

\subsection{Disjoint pairs: proof of Theorem B2}

For the other direction we need $J$ bounded \emph{above} in expectation, and here
the trivial estimate is useless: parent sizes can be as small as $2$. The upper
bound is where the multiplicative merger law \eqref{eq:mcrate} is used, and it
comes from a convexity inequality applied one step at a time.

\begin{lemma}[one-step partition inequality]
\label{lem:partition}
Let $k\ge2$ and let $a_1,\dots,a_k$ be positive integers with $\sum_i a_i=n$.
Then
\begin{equation}
\label{eq:partition}
\sum_{i<j}\frac{a_ia_j}{a_i+a_j}\ \le\ \frac{k}{2n}\sum_{i<j}a_ia_j .
\end{equation}
Equality holds if and only if $k=2$, or $a_1=\dots=a_k$.
\end{lemma}

\begin{proof}
Write $S_2=\sum_ia_i^2$ and $Z=\sum_{i<j}a_ia_j=(n^2-S_2)/2$. From the identity
\[
\frac{ab}{a+b}=\frac{a+b}4-\frac{(a-b)^2}{4(a+b)}
\]
and from $a_i+a_j\le n$ we get
\[
\sum_{i<j}\frac{a_ia_j}{a_i+a_j}
\ \le\ \frac14\sum_{i<j}(a_i+a_j)-\frac1{4n}\sum_{i<j}(a_i-a_j)^2 .
\]
Now $\sum_{i<j}(a_i+a_j)=(k-1)n$ and
$\sum_{i<j}(a_i-a_j)^2=(k-1)S_2-2Z=kS_2-n^2$, so the right-hand side equals
\[
\frac{(k-1)n}{4}-\frac{kS_2-n^2}{4n}=\frac{k(n^2-S_2)}{4n}=\frac{kZ}{2n},
\]
which is \eqref{eq:partition}.

For the equality case, the only inequality used was
$(a_i-a_j)^2/(a_i+a_j)\ge(a_i-a_j)^2/n$, which is an equality for the pair
$\{i,j\}$ exactly when $a_i=a_j$ or $a_i+a_j=n$. If $k=2$ then $a_1+a_2=n$ and
equality holds. If $k\ge3$ and two parts differ, say $a_i\ne a_j$, then equality
forces $a_i+a_j=n$, which is impossible since the remaining $k-2\ge1$ parts are
positive. Hence for $k\ge3$ equality holds exactly when all parts are equal, and
in that case both sides equal $\binom k2 n/(2k)$.
\end{proof}

Since the left side of \eqref{eq:partition} is, up to the normalisation $Z$,
exactly the conditional expectation of the increment of $J$, summing over the
states visited turns it into the bound we need.

\begin{lemma}[reciprocal budget]
\label{lem:budget}
For the accepted-merger coalescent on $K_n$,
\begin{equation}
\label{eq:EJ}
\E[J]\ \le\ \frac{(n-1)(n+2)}{4n},
\qquad\text{hence}\qquad
\E[H]\ \le\ \frac{(n-1)(5n+6)}{4n},
\end{equation}
with strict inequality in both for $n\ge4$, and equality for $n=2,3$.
\end{lemma}

\begin{proof}
By Lemma~\ref{lem:coalescent}, from a state with component sizes $a_1,\dots,a_k$
the increment of $J$ has conditional expectation
\[
\E[\Delta J\mid \text{state}]=\frac1Z\sum_{i<j}\frac{a_ia_j}{a_i+a_j}
\ \le\ \frac{k}{2n}
\]
by Lemma~\ref{lem:partition}. The process visits a state with exactly $k$
components once for each $k=n,n-1,\dots,2$, and $J$ is the sum of the increments
at those states, so by the tower property
\[
\E[J]\ \le\ \sum_{k=2}^n\frac k{2n}=\frac{n(n+1)/2-1}{2n}=\frac{(n-1)(n+2)}{4n} .
\]
The second bound in \eqref{eq:EJ} follows from Lemma~\ref{lem:HJ}:
\[
\E[H]=n-\frac1n+\E[J]\le\frac{4n^2-4+(n-1)(n+2)}{4n}=\frac{(n-1)(5n+6)}{4n} .
\]
For strictness, note that for $n\ge4$ the state reached after the first merger is
$(2,1,1,\dots,1)$ with $k=n-1\ge3$ components which are not all equal, so
Lemma~\ref{lem:partition} is strict there. For $n=2$ and $n=3$ every state
visited has $k=2$ or is the all-ones state, so every step is an equality; indeed
$J\equiv5/6$ and $H\equiv7/2$ deterministically at $n=3$.
\end{proof}

\begin{remark}[the merger law is needed, not only the state space]
\label{rem:notpathwise}
Unlike \eqref{eq:Jlower}, Lemma~\ref{lem:budget} has no pathwise version. For the
balanced merger history on $n=4$ leaves --- merge two disjoint pairs of
singletons, then merge the two blocks --- we have
\[
J=\frac12+\frac12+\frac14=\frac54\ >\ \frac{(n-1)(n+2)}{4n}=\frac98,
\]
so by \eqref{eq:PhiH} that history has $\E[\Phi\mid\text{history}]=4$, below the
threshold $3(n-1)(n-2)/n=4.5$ that the proof below requires. Nor would the state
space alone suffice: replacing the multiplicative choice \eqref{eq:mcrate} by a
uniform choice among the $\binom k2$ pairs of blocks breaks
Lemma~\ref{lem:partition} at once. At the state $(1,1,2)$, where $n=4$, $k=3$
and the bound $k/(2n)$ is $3/8$, the multiplicatively weighted average
$Z^{-1}\sum_{i<j}a_ia_j/(a_i+a_j)$ equals $11/30\le3/8$, while the unweighted
average of $1/(a_i+a_j)$ is $7/18>3/8$.
\end{remark}

\begin{proof}[Proof of Theorem B2]
Let $n\ge4$. Combining \eqref{eq:PhiH} with Lemma~\ref{lem:budget},
\begin{equation}
\label{eq:Philower}
\E[\Phi]=8(n-1)-4\,\E[H]\ >\ 8(n-1)-\frac{(n-1)(5n+6)}n=\frac{3(n-1)(n-2)}n .
\end{equation}
The tree has $(n-1)(n-2)$ ordered pairs of distinct edges, of which $\Phi$ are
adjacent; hence it has $(n-1)(n-2)-\Phi$ ordered disjoint pairs. The number of
ordered pairs of disjoint edges of $K_n$ is
$6\binom n4=n(n-1)(n-2)(n-3)/4$, and each contributes $p_2$, so
\[
p_2=\frac{4\bigl[(n-1)(n-2)-\E[\Phi]\bigr]}{n(n-1)(n-2)(n-3)}
\ <\ \frac{4(n-1)(n-2)\bigl[1-3/n\bigr]}{n(n-1)(n-2)(n-3)}
=\frac{4(n-3)}{n^2(n-3)}=\frac4{n^2}=p_0^2 . \qedhere
\]
\end{proof}

\begin{proof}[Proof of Corollary B3]
Two distinct edges of $K_n$ either share a vertex or not; apply Theorem B1 or
Theorem B2. For $n=3$ no disjoint pair exists.
\end{proof}

\section{The identity, and the reduction to \texorpdfstring{$\E[L_n]$}{E[Ln]}}
\label{sec:identity}

The bounds of \S\ref{sec:kn} were two estimates of the same functional $H$. That
functional has an exact meaning, and identifying it is what turns the pair
correlations on $K_n$ into a statement about a classical quantity. Recall that
$L_n$ is the weight of $\MST(K_n)$ under i.i.d.\ $\mathrm{Exp}(1)$ weights, so
that in the coalescent of \S\ref{sec:kn} the weights are the clock ring times.

\begin{lemma}
\label{lem:HeqnL}
$\E[H]=n\,\E[L_n]$ for every $n\ge2$.
\end{lemma}

\begin{proof}
Let $\kappa(t)$ be the number of components of the graph formed by the edges of
weight at most $t$. Since the weight of an edge is its ring time,
\[
L_n=\sum_{g\in T}w(g)=\int_0^\infty\#\{g\in T:w(g)>t\}\,dt
=\int_0^\infty(\kappa(t)-1)\,dt ,
\]
because at time $t$ the number of accepted mergers still to come is exactly
$\kappa(t)-1$. In the continuous-time process, at
a state with components of sizes $a_1,\dots,a_k$ the total accepted-merger rate
is $Z$ and the expected rate at which the functional $H$ accumulates is
\[
\sum_{i<j}a_ia_j\Bigl(\frac1{a_i}+\frac1{a_j}\Bigr)=\sum_{i<j}(a_i+a_j)=n(k-1),
\]
while $\kappa-1=k-1$. Thus $H_t-n\int_0^t(\kappa(s)-1)\,ds$ is a martingale.
Since $H\le2(n-1)$ and $L_n$ is integrable, letting $t\to\infty$ gives
$\E[H]=n\,\E[L_n]$.
\end{proof}

\begin{proof}[Proof of Proposition D1]
With $D$ the degree of a fixed vertex, $\E[D(D-1)]=\E[\Phi]/n$ by exchangeability
and $\E[D]=2(n-1)/n$ because every spanning tree has $n-1$ edges. Hence, by
\eqref{eq:PhiH} and Lemma~\ref{lem:HeqnL},
\[
\E[D^2]=\frac{\E[\Phi]}n+\frac{2(n-1)}n
=\frac{8(n-1)-4n\E[L_n]}{n}+\frac{2(n-1)}n
=\frac{10(n-1)}n-4\,\E[L_n]. \qedhere
\]
\end{proof}

\begin{proof}[Proof of Corollary D2]
For $n\ge3$, Proposition D1 and \eqref{eq:PhiH} show that
$p_1\le p_0^2$ is
$\E[\Phi]\le4(n-1)(n-2)/n$, i.e.\ $8(n-1)-4n\E[L_n]\le4(n-1)(n-2)/n$, i.e.\
$\E[L_n]\ge(n-1)(n+2)/n^2$. For $n\ge4$, likewise $p_2\le p_0^2$ is
$\E[\Phi]\ge3(n-1)(n-2)/n$, i.e.\ $\E[L_n]\le(n-1)(5n+6)/(4n^2)$.
Lemma~\ref{lem:harmonic}, Lemma~\ref{lem:HJ}, and Lemma~\ref{lem:HeqnL} give
\[
\E[L_n]\ge1+\frac{H_n-1}{n}-\frac1{n^2}.
\]
This implies the weaker lower threshold in the equivalence because
$H_n-1\ge(n-1)/n$. Equality in the harmonic bound holds for $n=2,3$; for
$n\ge4$ a non-comb history has positive probability, so it is strict. The
upper estimate and its equality statements follow from Lemma~\ref{lem:budget}.
\end{proof}

Thus the two pair inequalities become bounds on $\E[L_n]$, proved here without
input from the minimum-spanning-tree-weight literature. The harmonic lower bound
still tends to $1$ and is asymptotically loose. The upper bound tends to $5/4$:
it is exact at $n=2,3$, within $1\%$ of $\E[L_n]$ through $n=10$, and about
$4\%$ above the limit $\zeta(3)$. See Table~\ref{tab:exact}. Both pair ratios
increase in the computed range. The maximum
of $\E[L_n]$ in Gamarnik's range $2\le n\le45$ occurs at $n=8$, where
$\E[L_8]=1.2459812\ldots$ against the upper bound $161/128=1.2578125$;
decrease for every $n\ge8$ remains conjectural \cite[Conjecture 4.1]{Gamarnik}.

\begin{table}[t]
\centering
\begin{tabular}{rllcll}
\toprule
$n$ & $p_1/p_0^2$ & $p_2/p_0^2$ & $\E[L_n]$ & harmonic lower bd.
& upper bd.\ $\frac{(n-1)(5n+6)}{4n^2}$\\
\midrule
$3$  & $0.7500000000$ & ---            & $1.1667$ & $1.1667$ & $7/6$\\
$4$  & $0.7555555556$ & $0.9777777778$ & $1.2167$ & $1.2083$ & $39/32$\\
$5$  & $0.7597552910$ & $0.9804894180$ & $1.2353$ & $1.2167$ & $31/25$\\
$6$  & $0.7630719281$ & $0.9825707626$ & $1.2427$ & $1.2139$ & $5/4$\\
$7$  & $0.7657742674$ & $0.9842257326$ & $1.2454$ & $1.2071$ & $123/98$\\
$8$  & $0.7680285888$ & $0.9855771290$ & $1.2460$ & $1.1991$ & $161/128$\\
$9$  & $0.7699442296$ & $0.9867038469$ & $1.2455$ & $1.1909$ & $34/27$\\
$10$ & $0.7715965139$ & $0.9876591349$ & $1.2445$ & $1.1829$ & $63/50$\\
$11$ & $0.7730392810$ & $0.9884803595$ & $1.2432$ & $1.1754$ & $305/242$\\
$12$ & $0.7743121839$ & $0.9891945849$ & $1.2418$ & $1.1683$ & $121/96$\\
$13$ & $0.7754451621$ & $0.9898219351$ & $1.2405$ & $1.1618$ & $213/169$\\
$14$ & $0.7764613009$ & $0.9903777088$ & $1.2391$ & $1.1557$ & $247/196$\\
$20$ & $0.7809063153$ & $0.9927279258$ & $1.2323$ & $1.1274$ & $1007/800$\\
$30$ & $0.7850808025$ & $0.9948028441$ & $1.2250$ & $1.0987$ & $377/300$\\
\bottomrule
\end{tabular}
\caption{Exact pair correlations for $\MST(K_n)$, the harmonic lower bound,
and the upper bound from Corollary D2. Every entry comes from an exact rational;
the correlation-ratio, $\E[L_n]$, and harmonic-bound decimals are rounded as
displayed. The full exact table for $3\le n\le30$ accompanies the paper.}
\label{tab:exact}
\end{table}

Proposition D1 also recasts Gamarnik's conjectured decrease: for $n\ge8$,
$\E[L_n]>\E[L_{n+1}]$ is equivalent to
\[
\E[\dg_{n+1}(x)^2]-\E[\dg_n(x)^2]>\frac{10}{n(n+1)},
\]
where $\dg_k$ is the degree in $\MST(K_k)$.

\begin{remark}[what Theorem B2 says about the expected tree weight]
\label{rem:eln}
Corollary D2 exposes Theorem B2 as the non-asymptotic bound
\begin{equation}
\label{eq:elngate}
\E[L_n]\;\le\;\frac{(n-1)(5n+6)}{4n^2}\qquad(n\ge2),
\end{equation}
a statement about a quantity with a long history: Frieze \cite{Frieze}, Steele
\cite{Steele}, Janson \cite{Janson}, Fill and Steele \cite{FillSteele}, Gamarnik
\cite{Gamarnik}, Nishikawa, Otto and Starr \cite{NOS}, Cooper, Frieze, Ince,
Janson and Spencer \cite{CFIJS}, and Li and Zhang \cite{LiZhang}. These works
develop asymptotics, exact small-$n$ evaluations, and polynomial or Tutte
representations. None of these sources states the all-$n$ upper bound
\eqref{eq:elngate}. It is exact at $n=2,3$ and lies within
$0.0119$ of the largest value in Gamarnik's computed range, but its elementary
form makes prior knowledge plausible; Theorem B2 may therefore be a new
complete-range proof of a known bound rather than a new inequality. Theorem B1
uses the other half of Corollary D2 and is unaffected.
\end{remark}

\begin{proof}[Proof of Corollary D3]
The expansion of Cooper, Frieze, Ince, Janson and Spencer \cite{CFIJS} in
particular gives $\E[L_n^{\mathrm{unif}}]\to\zeta(3)$ for i.i.d.\
$\mathrm{Uniform}[0,1]$ weights. Li and Zhang \cite{LiZhang} give the
quantitative transfer
\begin{equation}
\label{eq:lizhang}
\E[L_n]-\E[L_n^{\mathrm{unif}}]=\frac{\zeta(3)}n
+O\Bigl(\frac{\log^2n}{n^2}\Bigr),
\end{equation}
so $\E[L_n]\to\zeta(3)$ for exponential weights as well. Insert this into
Proposition D1:
\[
\lim_{n\to\infty}\E_{\MST(K_n)}[\dg(x)^2]=10-4\zeta(3)=5.1917723873616\ldots
\]
By \cite[Proposition 3.11]{TZ} this is exactly $\E[N^2]$ for the root degree $N$
of the wired minimal spanning forest on the Poisson-weighted infinite tree
\cite{NT}, which is the assertion of \cite[Question 3.8]{TZ}.
\end{proof}

The sharper expansion identifies the rates and the two limiting correlation
ratios. Let $c_1,c_2$ be the constants in \cite[Theorem 1]{CFIJS}, so that
\begin{equation}
\label{eq:uniform-expansion}
\E[L_n^{\mathrm{unif}}]
=\zeta(3)+\frac{c_1}{n}+\frac{c_2+o(1)}{n^{4/3}},
\qquad c_1=0.0384956\ldots .
\end{equation}

\begin{corollary}[asymptotic rates]
\label{cor:rates}
For exponential weights,
\begin{align*}
\E[L_n]
&=\zeta(3)+\frac{c_1+\zeta(3)}n+\frac{c_2+o(1)}{n^{4/3}},\\
\E[\dg(x)^2]
&=10-4\zeta(3)-\frac{10+4(c_1+\zeta(3))}{n}
-\frac{4c_2+o(1)}{n^{4/3}}.
\end{align*}
Moreover,
\begin{align*}
\frac{p_1}{p_0^2}
&=2-\zeta(3)+\frac{4-c_1-4\zeta(3)}n
-\frac{c_2+o(1)}{n^{4/3}},\\
\frac{p_2}{p_0^2}
&=1+\frac{4\zeta(3)-5}{n}+O(n^{-2}).
\end{align*}
In particular the degree second moment converges at order $n^{-1}$, while the
adjacent and disjoint ratios tend to $2-\zeta(3)$ and $1$, respectively.
\end{corollary}

\begin{proof}
The first line follows by combining \eqref{eq:lizhang} and
\eqref{eq:uniform-expansion}; Proposition D1 gives the second. If
$m_n=\E[\dg(x)^2]$, direct counting of ordered adjacent and disjoint edge pairs
gives
\[
\frac{p_1}{p_0^2}
=\frac{m_n-2+2/n}{4(1-3/n+2/n^2)}\quad(n\ge3),\qquad
\frac{p_2}{p_0^2}
=\frac{n^2((n-1)-m_n)}{(n-1)(n-2)(n-3)}\quad(n\ge4).
\]
Substitution and expansion prove the last two formulas.
\end{proof}

This argument does not pass through the local weak limit. The local weak limit
of $\MST(K_n)$ is due to Addario-Berry
\cite{ABerry}, and convergence in distribution of the root degree does not by
itself give convergence of the second moment, which is why Question 3.8 was
asked; the recent study of dynamic local convergence for Prim's algorithm
\cite{CGH} likewise does not address moments. Proposition D1 replaces the
missing uniform integrability by an exact finite identity. Corollary
\ref{cor:rates} gives an asymptotic rate, although the $o(1)$ in
\eqref{eq:uniform-expansion} does not supply explicit finite error constants.
Tang and Zhang \cite[Proposition 3.11]{TZ}
print the decimal $5.191717$; the correct value of $10-4\zeta(3)$ is
$5.1917723873616\dots$, and the closed form there is not affected.

\section{Open problems}
\label{sec:open}

Theorem A locates positive correlation in the covariance of two bottleneck
distances, while Proposition D1 controls the complete-graph pair correlations
through the expected tree weight. Both descriptions leave natural quantitative
questions.

\begin{problem}[the optimal constant]
Determine the least $C$ for which the conclusion of Theorem A holds. We know only
that it lies in $[13938405/13872419,\,8]$. Corollary~\ref{cor:small-marginals} gives
the sharper factor $2+o(1)$ when both marginals vanish, but does not establish a
global factor $2$. The factor $8$ comes from the sharp inequality
$\E[T^2]\le2(\E T)^2$ and one factor $2$ from each deletion sandwich; improving
it requires using more of the joint structure than these three separate bounds.
\end{problem}

\begin{problem}[the simple-graph constant]
Determine the supremum of $p_{ef}/(p_ep_f)$ over finite connected simple graphs.
Proposition~\ref{prop:hub-witnesses} gives the lower bound
$13938405/13872419$, while Theorem A gives the upper bound $8$.
\end{problem}

\begin{problem}[effective errors and monotonicity]
Corollary~\ref{cor:rates} gives the asymptotic rate and the limits of both
complete-graph ratios, but not explicit finite error constants. Can the error in
\eqref{eq:uniform-expansion} be made effective enough to give useful finite-$n$
bounds? Both ratios increase for $3\le n\le30$ where defined
(Table~\ref{tab:exact}); are they monotone throughout their domains?
\end{problem}

\begin{problem}[other families at every size]
Tang and Zhang \cite[Remark 1.2]{TZ} state that their method also gives the
disjoint-pair property for $K_{n,n}$ for all sufficiently large $n$. Is there an
all-$n$ analogue of Corollary B3 for $K_{n,n}$, or for other natural families?
The method of \S\ref{sec:kn} uses the complete-graph transition rate
\eqref{eq:mcrate} and does not transfer as it stands. The minimum spanning tree
weight has been studied on several other host graphs, for instance regular graphs
\cite{BFM}, so the analogue of Corollary D2 there is a concrete starting point.
Finally, is there an analogue of Theorem A for the wired minimal spanning forest
on an infinite graph?
\end{problem}

\clearpage
\appendix

\section{Exact computations and computer assistance}
\label{sec:verification}

The universal statements of \S\ref{sec:universal} and the complete-graph
theorems of \S\ref{sec:kn}--\ref{sec:identity} use no computation. Computation
supplies the exact table and the finite enumerations in
Propositions~\ref{prop:simple-counterexample} and \ref{prop:hub-witnesses}, and
checks the algebra of Proposition~\ref{prop:bundle} and the coalescent model
against direct evaluations of $\PMST$. Every computed quantity is an exact
rational; decimals are rounded only when printed. The full table and
standard-library Python code are available in the public companion repository
\url{https://github.com/agupta/random-mst-correlations}; the citable preprint is
archived at \href{https://doi.org/10.5281/zenodo.21780630}{doi:10.5281/zenodo.21780630}.

\subsection{What each computation evaluates}

\begin{enumerate}
\item \emph{Direct Kruskal enumerations.} One route enumerates all $m!$ edge
orders; another sums exactly over the relative-order blocks determined by the
marked weights. These are distinct finite enumerations, but both use the same
Kruskal/union--find semantics. They cover $K_4$, $K_5$, the small-graph
examples, and the two- and three-hub instances; a symmetry-compressed version
handles the hub family for $1\le s\le9$.
\item \emph{The accepted-merger coalescent.} A dynamic program over integer
partitions with transition probabilities \eqref{eq:mcrate} returns $\E[J]$,
$\E[H]$, $\E[\Phi]$, $\E[D^2]$, $p_1$, $p_2$ and $\E[L_n]$ exactly. This
implements the proof in \S\ref{sec:kn}, rather than providing independent
evidence for it; the direct Kruskal computations test it in small cases.
\item \emph{$\E[L_n]$ from the area identity.} Independently of the coalescent,
the component count $\kappa(t)$ of $G(n,1-e^{-t})$ gives
\[
  \E[L_n]=\int_0^\infty(\E[\kappa(t)]-1)\,dt.
\]
Substituting $u=e^{-t}$ turns the integrand into an integer polynomial in $u$
built from the exact connectivity polynomial of $G(s,p)$, and the integral
becomes a finite rational sum. This route uses neither mergers nor degrees.
\end{enumerate}

The direct routes agree on every marginal and pair probability of $K_4$ and
agree with the coalescent program on $p_1$, $p_2$, and
$\E[\Phi]=8(n-1)-4\E[H]$ for $n=4,5$. The connectivity-polynomial route agrees
with it on $\E[L_n]$, thereby checking Proposition D1 without reusing the
coalescent identity $\E[H]=n\E[L_n]$. The resulting table agrees with
\cite{Gamarnik} through $n=30$; the committed transcription extends through
$n=45$. Internal identities such as $p_0=2/n$, $\sum_e p_e=n-1$, and
$\E[H]=n-\frac1n+\E[J]$ are also checked, but are not counted as independent
routes.

\subsection{What is verified, and in what range}

The documented exhaustive ranges are as follows. Merger histories are checked
through $n=12$, integer partitions through $n=32$, and the closing rational
algebra through $n<300$. Theorem A and its frozen-environment lemmas are checked
on every connected labelled multigraph with at most four vertices and at most
six edges, allowing loops, parallel classes, and disconnected deleted
environments. The bundle certificates are re-expanded coefficientwise and the
family is evaluated for $r+s\le120$; the $R(t+1)>R(t)$ certificate is similarly
re-expanded and scanned through $t=300$. The hub-family computation covers
$1\le s\le9$. The simple-graph minimality assertion uses the bridge reduction
in the text and a census of 1528 labelled bridgeless graphs with at most seven
edges, reducing to 17 isomorphism classes and 272 representative marked pairs.
Theorem C and Remark~\ref{rem:attenuation} are also checked directly for
subdivision lengths $t=1,2,3$.

Several tempting strengthenings fail, and the suite retains them as controls:
$C=1$ fails by \eqref{eq:lps}; $\E[T^2]\le\frac32(\E T)^2$ fails for the
minimum of six uniforms; $a\ge\frac34A$ fails when the marked edges join the
same two vertices and no other edge is present; Lemma~\ref{lem:harris} is an
equality when both marked edges are bridges; and Lemma~\ref{lem:partition}
fails at $(2,1,1)$ if block pairs are chosen uniformly rather than with
multiplicative weights.

These checks establish only the stated finite assertions and test the modelling
and algebra; they do not replace the analytic proofs. Complete commands and
ranges are recorded in the companion repository README.

\section*{Acknowledgment of generative-AI assistance}
Anthropic Claude Code (Claude 5 family) and OpenAI Codex (GPT-5.6 family) were
used extensively for proof exploration, software development, exact
computational checks, literature discovery, and drafting and editing the
manuscript. The author selected the arguments and methods, checked the cited
sources and reported computations, and takes full responsibility for the
content. These systems are not authors or independent guarantors of
correctness.

\section*{Declarations}
The author declares no competing interests. This research received no external
funding. The manuscript and accompanying data are released under CC BY 4.0;
the accompanying software is released under the MIT License.

\end{document}